\documentclass[11pt,twoside,a4paper]{amsart}
\usepackage{amssymb,amsmath,latexsym,xypic,amscd,amsthm}
\usepackage{anysize}\marginsize{3.5cm}{3.5cm}{1.3cm}{2cm}
\usepackage{graphicx}
\newtheorem{satz}{Satz}[section]

\newtheorem{corollary}[satz]{Corollary}
\newtheorem{definition}[satz]{Definition}

\newtheorem{lemma}[satz]{Lemma}
\newtheorem{proposition}[satz]{Proposition}
\newtheorem{remark}[satz]{Remark}

\newtheorem{theorem}[satz]{Theorem}
\newtheorem{conjecture}[satz]{Conjecture}

\newcommand\up[1]{\left\lceil #1 \right\rceil}

\newcommand{\F}{\mathcal{F}}

\newcommand{\conv}{{\rm conv}}

\renewcommand\dim{\mathrm{dim}}

\renewcommand\emptyset{\varnothing}
\renewcommand\ge{\geqslant}

\renewcommand\geq{\geqslant}
\renewcommand\leq{\leqslant}
\renewcommand\epsilon{\varepsilon}
\renewcommand\phi{\varphi}

\renewcommand\P{\mathbb P}

\newcommand\ehr{{\rm ehr}}
\newcommand\Vol{{\rm Vol}}

\newcommand\lolra{\Longleftrightarrow}
\newcommand\lora{\Longrightarrow}

\newcommand\aff{{\rm aff}}

\newcommand\cd{{\rm codeg}}

\newcommand\N{\mathbb N}
\newcommand\Q{\mathbb Q}
\newcommand\R{\mathbb R}
\newcommand\Z{\mathbb Z}

\begin{document}

\title[Projective Toric Manifolds with Dual Defect]{A Simple Combinatorial Criterion for Projective Toric
Manifolds with Dual Defect}
 
\author{Alicia Dickenstein}
\address[Alicia Dickenstein]{
   Departamento De Matem\'atica, FCEN, Universidad de Buenos Aires, C1428EGA Buenos Aires, Argentina
}
\thanks{AD was
partially supported by UBACYT X064, CONICET PIP 112-200801-00483 and ANPCyT PICT
20569, Argentina}
\email{
   alidick@dm.uba.ar
}

\author{Benjamin Nill}
\address[Benjamin Nill]{
   Department of Mathematics, University of Georgia, Athens, GA 30602, USA
}
\thanks{BN was supported by fellowship HA 4383/1 of 
the German Research Foundation (DFG) as a member of the Emmy Noether research group Lattice Polytopes led by Christian Haase, and by 
an MSRI postdoctoral fellowship as part of the Tropical Geometry program.}
\email{
   bnill@math.uga.edu
}

\begin{abstract}
We show that any smooth lattice polytope $P$ with codegree greater or equal than
$(\dim(P)+3)/{2}$ (or equivalently, with degree smaller than ${\dim (P)}/{2}$),
defines a dual defective projective toric manifold. This implies that $P$ is $\Q$-normal
(in the terminology of \cite{DDP09}) and answers partially an adjunction-theoretic conjecture by
Beltrametti-Sommese \cite{BS95} and \cite{DDP09}. Also, it follows from \cite{DR06} that smooth
lattice polytopes with this property are precisely strict Cayley polytopes, which
completes the answer in \cite{DDP09} of a question in \cite{BN07} for smooth polytopes.
\end{abstract}
\maketitle
\section{Introduction}

\subsection{Cayley polytopes and (co)degree}

Let $P \subset \R^n$ be a lattice polytope of dimension $n$. Given a positive integer
$k$, we denote by $kP$ the lattice polytope obtained as the Minkowski sum of $k$ copies of $P$,
and by $(kP)^{\circ}$ its interior. The {\em codegree} of $P$ is the following invariant:
\[\textstyle \cd(P):=\min \{ k\, | \, (kP)^{\circ} \cap \Z^n \neq\emptyset\}.\]
The number $\deg(P) := n+1-\cd(P)$ is called the {\em degree} of $P$, see \cite{BN07}. 

It has recently been proven \cite{HNP08} that, if the codegree of $P$ is large with respect to $n$, then $P$ lies between two 
adjacent integral hyperplanes (i.e., its lattice width is one). This gave a positive answer to a question of V.V. Batyrev and 
the second author \cite{BN07}. Actually, in \cite{HNP08} a stronger statement was proven.
For this, let us recall the notion of a (strict) Cayley polytope, see \cite{HNP08, DDP09}.

\begin{definition}{\rm Let $P_0, \ldots, P_k \subset \R^m$ be lattice polytopes
such that 
the dimension of the affine span $\aff(P_0, \ldots, P_k)$ equals $m$. 
Then we define the {\em Cayley polytope} of $P_0, \ldots, P_k$ as
\[P_0 * \cdots * P_k := \conv(P_0 \times e_0, \ldots, P_k \times e_k) \subseteq \R^m \oplus \R^{k+1},\]
where $e_0, \ldots, e_k$ is a lattice basis of $\R^{k+1}$. Note that it is a lattice polytope 
of dimension $m+k$. If additionally all $P_0, \ldots, P_k$ are {\em strictly isomorphic}, i.e., 
they have the same normal fan, then we call $P_0 * \cdots * P_k$ a {\em strict Cayley polytope}. 
\label{def-cay}
}
\end{definition}

It was observed \cite{BN07} that in this situation $\cd(P) \geq k+1$ holds, or equivalently, $\deg(P) \leq m$. 
Now, in \cite{HNP08} it was shown that there is also a partial converse. Namely, if $n > f(\deg(P))$, where $f$ is a quadratic polynomial, then $P$ 
is a Cayley polytope of lattice polytopes in $\R^{f(\deg(P))}$. It is believed that this is not a sharp bound, 
there is hope that $f(\deg(P))$ may be simply replaced by $2\deg(P)$. 

\begin{conjecture}{\rm Let $P \subset \R^n$ be an $n$-dimensional lattice polytope. 
If $\cd(P)\geq \frac{n+3}{2}$ (or equivalently, $n > 2 \deg(P)$), 
then $P$ is a Cayley polytope of lattice polytopes $P_0, \ldots, P_k \subseteq \R^m$, where 
$m \leq 2 \deg(P)$.
\label{cay-conj}
}
\end{conjecture}

Note that we cannot expect to get better than this, since there exist $n$-dimensional lattice simplices $P$ for $n$ even with 
$n = 2 \deg(P)$ that are not Cayley polytopes (for any $k \geq 1$), see Example 1.6 in \cite{DDP09}. 

Now, recall that a lattice polytope $P$ of dimension $n$ is called {\em smooth}, if 
there are exactly $n$  facets intersecting at any vertex of $P$ and
 the primitive inner normals to these facets are a basis of $\Z^n$. 
As a consequence of the main result of our paper, Theorem~\ref{th:main}, we can settle this important case:

\begin{corollary}
Conjecture~\ref{cay-conj} holds for smooth lattice polytopes.
\label{smoothy}
\end{corollary}

Actually, Theorem~\ref{th:main} shows much more. 
For this, let us remark that if $P$ is a Cayley polytope given as in Definition \ref{def-cay}, 
then $k > n/2$ implies $\cd(P) > n/2 + 1$, so $\cd(P) \geq (n+3)/2$. 
Now, Theorem~\ref{th:main} yields that in the smooth case also the converse is true. 
Namely, if $P$ is a smooth $n$-dimensional lattice polytope $P$ with 
$\cd(P) \geq (n+3)/2$, then $P$ is a Cayley polytope of $k+1$ lattice polytopes in dimension $m = \deg(P)$ 
(or equivalently, $k = \cd(P) - 1$). In particular, $k > n/2$. 

In the singular case, we cannot expect such a strong statement, as we see from the following example. 
Let $Q$ be the convex hull of $(0,0),(2,0),(0,2)$. 
This is a lattice triangle of degree $1$. Since taking lattice pyramids does not change the degree \cite{BN07}, 
the three-fold lattice pyramid $Q'$ over $Q$ is still a five-dimensional lattice simplex of degree $1$. 
Finally, let $P$ be defined as $Q' \times [0,2]$. This is a simple lattice polytope of dimension $n=6$, degree $2$ and codegree $5$. 
Hence, $\cd(P) \geq (n+3)/2$. We see that $P$ is a Cayley polytope of four lattice polytopes in dimension $m=3$. 
Here $k=3$ is maximal, yet, $k \not> n/2$.

\subsection{$A$-discriminants and dual defect toric varieties}

Given a configuration $A = \{a_1, \dots, a_N\}$ of  lattice points in $\R^n$
one gets an associated {\em projective toric variety} $X_A \subset \P^{N-1}$, 
rationally parametrized from the $n$-th torus
by $(t_1, \dots, t_n) \mapsto (t^{a_1}: \dots : t^{a_N})$.
The dimension of $X_A$
equals  the dimension of the affine span of the points in $A$ and, in fact, $X_A = X_{A'}$ for any lattice
configuration $A'$  affinely isomorphic to $A$ (that is, $A'$ is the image of $A$ by an injective affine linear transformation).  
The {\em dual variety} $X_A^*$ is defined as the Zariski closure of the locus of hyperplanes in $(\P^{N-1})^*$ which are tangent to a smooth point of $X_A$. 
Generically, $X_A^*$ is a hypersurface whose defining equation (defined up to sign) is 
called the {\em A-discriminant} $\Delta_A$. We call $X_A$ a {\em dual defect variety}, 
if $X_A^*$ is not a hypersurface, in which case we set $\Delta_A := 1$. The {\em dual defect} of $X_A$ is defined 
as $N-2-\dim(X_A^*)$. There is a vast literature on the study of $A$-discriminants starting 
with the seminal work of Gel'fand, Kapranov and Zelevinsky \cite{GKZ94}. 

Here, we focus on the case, where 
$A = P \cap \Z^n$ for a smooth $n$-dimensional lattice polytope $P \subset \R^n$. 
In this case, $X_A$ is a projective toric manifold. It is isomorphic to the 
abstract toric variety associated to the normal fan of $P$ 
via the projective embedding given by the very ample line bundle corresponding to $P$ 
(see \cite{Ful93} for standard results in toric geometry). 
%Open whether (proj.) normal!!
Since $X_A$ is smooth, the degree of $\Delta_A$ can be computed by the following combinatorial invariant $c(P)$, 
see Chapter 9 of \cite{GKZ94} or \cite{DR06}.  Let us denote by $\F(P)_j$ the set of $j$-dimensional faces of $P$. Then
\[c(P) := \sum_{j=0}^n (-1)^{n-j} (j+1) \sum_{F \in \F(P)_j} \Vol_\Z(F),\]
where $\Vol_\Z(F)$ is the {\em normalized volume} of $F$, defined such that the fundamental parallelepiped of $\aff(F) \cap \Z^n$ 
has normalized volume $\dim(F)!$. In particular, 
the normalized volume of any lattice polytope is a natural number. 
For formulas of $\deg(\Delta_A)$ in the general (singular) case we refer to \cite{CC07, DFS07, Est08, MT08}.

In particular, $X_A$ has dual defect if and only if $c(P) = 0$. 
Our main result, Theorem \ref{th:main}, shows that there is a simple equivalent condition purely in terms of the codegree of 
the smooth polytope $P$.

\subsection{The relation to the adjunction theory of polarized varieties}

Let us assume that $(X,L)$ is a {\em polarized manifold}, i.e., 
$X$ is a projective manifold with ample line bundle $L$ and canonical bundle $K_X$.
In the adjunction theory of polarized varieties (we refer to the book \cite{BS95}) 
there are two invariants which have been studied quite intensively: 
The {\em (unnormalized) spectral value} $\mu$, which is the supremum of all $t$ such that 
$t L + K_X$ does not have global sections. And the {\em nef-value} $\tau$, which equals the infimum of all $t$ 
such that $t L + K_X$ is not nef (i.e., numerically-effective). Using the notation of \cite{DDP09}, let us 
call $(X,L)$ {\em $\Q$-normal}, if these two invariants coincide. The following conjecture \cite[Conj.7.1.8]{BS95} has been confirmed for $n \leq 7$ 
and in many other cases.

\begin{conjecture}[Beltrametti, Sommese] 
If an $n$-dimensional polarized manifold $(X,L)$ satisfies $\mu > \frac{n+1}{2}$, then it is $\Q$-normal.
\label{bs-conj}
\end{conjecture}

If $(X,L)$ is given by a smooth lattice polytope $P$, then $\mu$ and $\tau$ can be computed purely in terms of $P$, 
and we have $\mu \leq \cd(P)$, see \cite{DDP09}. Hence, Theorem~\ref{th:main} has the following corollary.

\begin{corollary}
If an $n$-dimensional polarized toric manifold $(X,L)$ satisfies $\mu > \frac{n+2}{2}$, then it is $\Q$-normal.
\end{corollary}

In particular, Conjecture \ref{bs-conj} holds, except possibly if $n$ is even and $(n+1)/2 < \mu  \leq (n+2)/2$.

\subsection{The main result}

The goal of this paper is to complete the proof of the following theorem: 

\begin{theorem}\label{th:main}
Let $P\subset \R^n$ be a smooth lattice polytope of dimension $n$. 
Then the following statements are equivalent:

\begin{enumerate}
\item[{\rm (i)}] $\cd(P)\geq \frac{n+3}{2}$,   
\item[{\rm (ii)}] $P$ is affinely isomorphic to a strict Cayley 
polytope $P_0 * \ldots * P_k$, where $k+1=\cd(P)$ with $k>\frac{n}{2}$
(and $\dim(P_j) = \deg(P)$ for all $j = 0, \dots, k$),
\item[{\rm (iii)}]  the toric polarized variety $(X,L)$ corresponding to
$P$ is dual defective (necessarily, with dual defect $\delta = 2 \cd (P) -  2 - n$),
\item[{\rm (iv)}] $c(P) = 0$.
\end{enumerate}
If these conditions hold, then $(X,L) $ is $\Q$-normal with $\mu=\cd(P) =\tau$.
\end{theorem} 

In particular, (ii) implies that $P$ is combinatorially simply a product of a $(\cd(P)-1)$-simplex and 
a $\deg(P)$-dimensional polytope. Most parts of the proof have already been done by S. Di Rocco \cite{DR06} 
and by S. Di Rocco, R. Piene and the first author \cite{DDP09}. The only new ingredient 
is Theorem~\ref{main-thm}, which is used to remove 
the $\Q$-normality assumption in \cite{DDP09}. 

\begin{proof}
We have (i) $\lora$ (iv) by Theorem~\ref{main-thm}(i) below. 
Propositions 2.3 of \cite{DR06} shows that (iv) holds if and only if 
\[X \cong \P(L_0 \oplus \cdots \oplus L_{\frac{n+\delta}{2}}),\]
where $L_i$ are ample line bundles on a toric manifold of dimension $\frac{n-\delta}{2}$, 
where $\delta \geq 1$ is the dual defect of $X \subseteq \P^{N-1}$. 
An equivalent formulation was given in Proposition 3.7 of \cite{DR06} as follows:
\begin{equation}
P \cong P_0 * \cdots * P_{\frac{n+\delta}{2}},
\label{crit1}
\end{equation}
for strictly isomorphic smooth $\frac{n-\delta}{2}$-dimensional lattice polytopes $P_i$ ($i=0, \ldots, \frac{n+\delta}{2}$), 
where $\delta \geq 1$. 
Note that this immediately implies (iv) $\lora$ (i). Moreover, 
Proposition 3.9 in \cite{DDP09} yields that (\ref{crit1}) implies $\Q$-normality. 
Therefore, we can apply the main theorem, Theorem 1.12, in \cite{DDP09} to deduce that (i) $\lora$ (ii) 
$\lolra$ (iii) $\lora$ (iv), as well as $\mu=\cd(P)=\tau$.
\end{proof} 

\subsection{An illustration of our result}

Let us determine when the 
product of projective spaces embedded by Segre and Veronese embeddings has dual defect. 
This is Corollary 5.11 of Chapter 1 of \cite{GKZ94}. This reproves also the criterion 
for the existence of the multigraded discriminant (see Chapter 13, Proposition 2.3) 
and for the non-triviality of the hyperdeterminant (see Chapter 14, Theorem 1.3) 
described in \cite{GKZ94}.

We only need the following simple corollary of Theorem~\ref{th:main}:

\begin{corollary}
Let $P_1, \ldots, P_r$ be smooth lattice polytopes. 
Then the toric variety associated to $P_1 \times \cdots \times P_r$ has dual defect if and only if
\begin{equation}
2 \max(\cd(P_1), \ldots, \cd(P_r)) \geq \dim(P_1) + \cdots + \dim(P_r)+3.
\label{criterion}
\end{equation}
\end{corollary}

\begin{proof}
Note that the definition of the codegree immediately implies 
$\cd(P_1 \times P_r) = \max(\cd(P_1), \ldots, \cd(P_r))$.
\end{proof}

Let $S_n$ be again the $n$-dimensional unimodular simplex, i.e., $\Vol_\Z(S_n) = 1$. 
It is easy to see that $\cd(S_n) = n+1$. 
Let us define simplices $P_i := d_i S_{k_i}$ for positive natural numbers $d_1, \ldots, d_r$, $k_1, \ldots, k_r$. 
In this case, we can reformulate the criterion (\ref{criterion}) as 
\[2 \max\left(\up{\frac{k_1+1}{d_1}}, \ldots, \up{\frac{k_r+1}{d_r}}\right) \geq k_1 + \cdots + k_r+3.\]
This condition can only  be satisfied, if the maximum is attained for $i \in \{1, \ldots, r\}$ with $d_i = 1$. 
Hence, we get that the toric variety associated to $P_1 \times \cdots \times P_r$ has dual defect 
if and only if 
\[2 k_i > k_1 + \cdots + k_r \text{ for some } i \in \{1, \ldots, r\} \text{ such that } d_i = 1.\] 
For instance, let us consider the simplest possible case: $r=2$, $d_1=d_2=1$. Here, we 
deduce that $\P^{k_1} \times \P^{k_2}$ embedded by the Segre embedding has dual defect 
if and only if $k_1 = k_2$. In Chapter 9, Example 2.10(b), of \cite{GKZ94} this was reproven by explicitly calculating 
the invariant $c(S_{k_1} \times S_{k_2})$ using a nice observation on sums of binomial coefficients (Chapter 9, Lemma 2.9):
\[c(S_{k_1} \times S_{k_2}) = \sum_{i=1}^{k_1+1} \sum_{j=1}^{k_2+1} 
(-1)^{k_1+k_2-i-j} (i+j-1) 
\binom{k_1+1}{i} \binom{k_2+1}{j} \binom{i+j-2}{i-1}.\]
We see that already in this situation checking $\cd(P) \geq (\dim(P)+3)/2$ turns out to be much simpler 
than computing $c(P)$. 
However, binomial identities of this flavour will come up in the proof of our main result.

\subsection{Organization of the paper}

This paper is organized as follows. In the second section we prove our main result Theorem \ref{main-thm}, which 
gives a non-trivial relation among the volumes of the faces of a simple lattice polytope $P$, if $\cd(P)>1$. 
In the third section we consider the conjecture of S. Di Rocco stating that $c(P) \geq 0$ for any lattice polytope $P$, which we 
confirm for lattice simplices. In the last section, we sketch some directions for future research.

\medskip

\textbf{Acknowledgments:} This research started during the workshop ``Combinatorial challenges in toric varieties'' 
(April 27 to May 1, 2009) 
at the American Institute of Mathematics (AIM) and was finished at the Mathematical Sciences Research Institute (MSRI), where both authors stayed 
as research members of the Tropical Geometry program (August 17 to December 18, 2009). 
We thank Sandra di Rocco for useful discussions.

\section{The main result}

Throughout the rest of the paper, let $n \geq 2$ and $0 \leq d \leq n$.

\subsection{The statement of the theorem}

Recall that an $n$-dimensional polytope $P$ is {\em simple}, if every vertex is contained in precisely $n$ facets. 
For instance, smooth polytopes are simple. Here is our main combinatorial result:

\begin{theorem}
Let $P$ be an $n$-dimensional simple polytope of degree $d < n$. 
\begin{enumerate}
\item[{\rm (i)}] If $d < n-d$ (equivalently, $d \leq \frac{n-1}{2}$, respectively $\cd(P) \geq \frac{n+3}{2}$), then
\[c(P)=0.\]
\item[{\rm (ii)}] If $d \geq n-d$, then
\[\sum_{j=0}^n (-1)^j \sum_{F \in \F(P)_j} ((n-d) h^*_{n-d}(F) + (j+1) \sum_{k=0}^{n-d-1} h^*_k(F)) = 0.\]
\end{enumerate}
\label{main-thm}
\end{theorem}

The notation in the second statement is explained in the next subsection. 
The proof itself is given in Subsection~\ref{pr:main}. Since it involves 
basic Ehrhart theory and some identities of Binomial coefficients, we will 
discuss these topics first. The ideas of the proof are outlined in Subsection~\ref{pr:ideas}.

\subsection{Basics of Ehrhart theory}

Let us start with recalling some standard notions and results in Ehrhart theory, see \cite{BR06}. 

Let $P \subseteq \R^n$ be an $n$-dimensional lattice polytope. The {\em Ehrhart polynomial} $\ehr_P$ 
is given by the function ($k \in \N \mapsto |(k P) \cap \Z^n|$). 
It has a well-known rational generating function \cite{Ehr77, Sta80}:
\[\sum\limits_{k \geq 0} |(k P) \cap \Z^n| \, t^k =  \frac{h^*_0 + h^*_1 t + \cdots + h^*_n t^n}{(1-t)^{n+1}},\]
where $h^*_0, \ldots, h^*_n$ are non-negative integers 
satisfying $h^*_0 + \cdots + h^*_n =\Vol_\Z(P)$. The enumerator polynomial $h^*_P(t)$ is sometimes called the 
{\em $h^*$-polynomial} of $P$. Its degree (i.e., the maximal $k$ such that $h^*_k \not=0$) equals the 
degree of $P$, see \cite{BN07}. If $F$ is a face of $P$, then Stanley's monotonicity theorem \cite{Sta93} 
yields $\deg(F) \leq \deg(P)$, which we will use later on.

Let us also remark that switching between the Ehrhart polynomial and the $h^*$-polynomial 
is merely a linear transformation corresponding to the choice of a basis of binomial coefficient polynomials 
instead of a monomial basis:
\begin{equation}
\ehr_P(t) = \sum_{k=0}^{\deg(P)} h^*_k(P) \binom{t+n-k}{n}.
\label{h-st}
\end{equation}
The Ehrhart polynomial also allows to count the number of interior lattice points. 
For $k \in \N_{\geq 1}$ we have $|(k P)^\circ \cap \Z^n| = (-1)^n \ehr_P(-k)$, which is called Ehrhart reciprocity. 
In particular, it shows that counting the number of interior lattice points is also polynomial.

\subsection{Identities of binomial coefficients}

%Error in GKZ94

The core of the proof is a calculation involving alternating sums of products of binomial coefficients. 
It is interesting to note that similar formulas can be found in related work: 
naturally, when determining the degree of the $A$-discriminant in the smooth case 
(pp. 282-285 in \cite{GKZ94}), in explicit generalized Bott formulas for toric varieties \cite{Mat02}, 
and more recently in the computation of 
the $h$-vector of the regular triangulation of a hypersimplex in order to bound 
the $f$-vectors of the tight span \cite{HJ07}.

Let us recall the following convolution formula. A proof can be found in 
Gr\"unbaum's classical book on polytopes \cite[p. 149]{Gru03}. 
Interestingly, we found this reference in a recent paper about tropical intersection curves and Ehrhart theory 
\cite{ST09}.

\begin{lemma}
Let $0 \leq c \leq a$ and $0 \leq b$, then
\[\sum_{q=0}^c (-1)^q \binom{b}{q} \binom{a-q}{c-q} = \binom{a-b}{c}.\]
\label{binom}
\end{lemma}

Here is the main lemma. In contrast to the other above cited papers, 
we do not (yet) provide a direct proof, instead we verify this identity using the amazing Zeilberger's algorithm, 
for which we refer to the ``$A=B$'' book \cite{PWZ96}.

\begin{lemma}
For $k \in \{0, \ldots, d\}$ and $j \in \{k, \ldots, n\}$ we have the following identities:
\[\sum_{i=0}^{n-d} (-1)^{n-d-i} i \binom{i+j-k}{j} \binom{j+1}{n-d-i} 
= \left\{\begin{array}{rcl}j+1 & , & k < n-d\\n-d & , & k=n-d\\0 &, & k > n-d\end{array}\right.\]
\label{formulas}
\end{lemma}

\begin{proof}
For $a := n-d \geq 1$, let us define the function
\[f(k) := \sum_{i=0}^\infty F(k,i),\]
where
\[F(k,i) := (-1)^{a-i} \, i \binom{i+j-k}{j} \binom{j+1}{a-i}.\]
Let us show
\begin{equation}
f(k+1) = f(k), \text{ if } k \not\in \{a-1,a\}.
\label{claim}
\end{equation}

To see this claim, we apply the function \texttt{zeil} of the \texttt{maple}-package \texttt{EKHAD8}. This returns 
the following equation:
\begin{equation}
(a - k) (a - 1 - k) (j + 1) (F(k,i) - F(k+1,i)) = G(k,i+1) - G(k,i),
\label{eq1}
\end{equation}
where $G(k,i) := F(k,i) \cdot R(k,i)$ for the rational function
\[R(k,i) := (-j - 1 + a - i) (i-k) (k a - i k j + a j i - k i - a j)/((i + j - k) i).\]
%Note that $G(k,i)$ evaluates to a well-defined number. 
We remark that the reader can verify the algebraic equation 
(\ref{eq1}) directly after dividing both sides by $F(k,i)$. Now, summing over $i$ on both sides of Equation (\ref{eq1}) yields 
\begin{equation}
(a - k) (a -1 - k) (j + 1) (f(k)-f(k+1)) = -G(k,0),
\label{eq2}
\end{equation}
where
\[-G(k,0) = (-1)^{a+1} \binom{j-k}{j} \binom{j+1}{a} (-j - 1 + a) k a.\]
Using our assumption $j-k \geq 0$ we see that $-G(k,0) = 0$. 
Hence, Equation (\ref{eq2}) yields that $f(k)-f(k+1) = 0$, if $a-k \not=0$ and $a-1-k \not=0$, which proves our claim (\ref{claim}). 
Let us distinguish three cases:

\begin{enumerate}
\item[{\rm (i)}] Let $k < a$. In this case, (\ref{claim}) yields $f(k) = f(k-1) = \cdots = f(0)$. Now, we apply $\texttt{zeil}$ again 
on $f(0)$ as a function in $j$:
\[f_1(j) := \sum_{i=0}^\infty (-1)^{a-i} \, i \binom{i+j}{j} \binom{j+1}{a-i}.\]
In the same way as above we get the recurrence equation 
\[(-j-2) f_1(j) + (j+1) f_1(j+1) = 0.\]
Since $j \mapsto j+1$ satisfies the same recurrence equation and 
the initial values $f_1(0) = a - (a-1) = 1$ coincide, this shows $f(0) = f_1(j) = j+1$, as desired.
\item[{\rm (ii)}] Let $k = a$. Here we regard $f(a)$ as a function in $j$:
\[f_2(j) := \sum_{i=0}^\infty (-1)^{a-i} \, i \binom{i+j-a}{j} \binom{j+1}{a-i}.\]
Again by \texttt{zeil} we get that $f_2(j)-2 f_2(j+1)+f_2(j+2)$ equals
\[\frac{(-1)^{a+1} \binom{j-a}{j} \binom{j+1}{a} a (-2 aj+a^2-4a+j^2+4j+3)}{(j+1)(a-j-2)}.\]
Since $a=k\leq j$, this expression evaluates to $0$. Therefore, $f_2$ satisfies for $j \geq a$ 
the same recursion as a constant function. Now, we only have to observe that 
$f_2(a) = a$ and $f_2(a+1) = a$ to get $f(a) = f_2(j) = a$.
\item[{\rm (iii)}] Let $k > a$. In this case, (\ref{claim}) yields $f(k) = f(k-1) = \cdots = f(a+1)$. 
Now, we apply $\texttt{zeil}$ again on $f(a+1)$ as a function in $j$:
\[f_3(j) := \sum_{i=0}^\infty (-1)^{a-i} i \binom{i+j-a-1}{j} \binom{j+1}{a-i}.\]
As above, we get $(-j-2) f_3(j) + (j+1) f_3(j+1) = 0$ for $j \geq k=a+1$. 
Again, it remains to observe that $f_3(a+1) = 0$.
\end{enumerate}
\end{proof}

\subsection{The idea of the proof}

\label{pr:ideas}

Let us give an outline of the proof that $\cd(P) \geq (n+3)/2$ implies $c(P) = 0$. 
By definition, conditions on the codegree of $P$ translate into the vanishing 
of certain values of the Ehrhart polynomial of $P$. Using a well-known inclusion-exclusion formula 
for simple polytopes (e.g., see Exercise 5.9 in \cite{BR06}) 
we get the following equations:
\begin{equation}\label{eq:vanishing}
0 = |(k P) ^\circ \cap \Z^n| = \sum_{j=0}^n (-1)^j \sum_{\dim F=j} |(k F) \cap \Z^n | \quad \forall\; k = 1, \dots, \lceil\frac{n+1}{2}\rceil,
\end{equation}
where the sum is over all faces $F$ of $P$. One motivation to use these formulas 
is their interpretation as the dimensions of the cohomology
groups $H^0(X, \Omega_X^n (kL))$ ($k = 1, \dots, \lceil\frac{n+1}{2}\rceil$), see \cite{Mat02}, which is also 
valid for $\Q$-factorial toric varieties. 

Now, we expand these linear equations in terms of the coefficients of the $h^*$-polynomial. 
In order to have as much linear equations as possible 
we consider all the faces of $P$ which have interior lattice points. 
On the other hand, we also observe that $c(P)$ can be easily expressed as an expression which is linear 
in terms of the coefficients of the $h^*$-polynomials. 
So, the theorem would follow, if it would turn out that 
$c(P)$ is a linear combination of the equations we started with. Amazingly, this is true. 
Indeed, the proof is a purely formal combinatorial argument, which does not involve 
any (non-trivial) geometry. The lucky part is to come up with the right (binomial) coefficients, 
which we guessed based upon low-dimensional experiments. 
We do not know yet of a more insightful and systematic way to prove our result.

\subsection{Proof of Theorem~\ref{th:main}}

\label{pr:main}

We are going to show that both  expressions in the statement of the theorem are actually equal to
\begin{equation}
\sum_{p=d+1}^n \sum_{i=1}^{p-d} (-1)^{d-i} i \binom{p+1}{p-d-i} \left(\sum_{G \in \F(P)_p} 
|(i G)^\circ \cap \Z^n|\right).
\label{expr-proof}
\end{equation}

Let us first note that this is indeed zero, as desired: 
If $G \in \F(P)_p$ for $p > d$, then 
Stanley's monotonicity theorem implies $1 \leq p-d \leq p-\deg(G) = \cd(G)-1$. 
In particular, $|(i G)^\circ \cap \Z^n| = 0$ for any $i = 1, \ldots, p-d$, hence, the formula vanishes.

Let us now fix $p \in \{d+1, \ldots, n\}$ and $i \in \{1, \ldots, p-d\}$. 
Let $G \in \F(P)_p$. The inclusion-exclusion-formula for simple polytopes yields:
\[\sum_{G \in \F(P)_p} |(i G)^\circ \cap \Z^n| = 
\sum_{G \in \F(P)_p} \left(\sum_{j=0}^p (-1)^{p-j} \sum_{F \in \F(G)_j} \ehr_F(i)\right).\]
Moreover, since $P$ is simple, any face $F$ of $P$ of dimension $j \leq p$ 
is contained in precisely $\binom{n-j}{n-p}$ faces $G$ of $P$ of dimension $p$. This yields
\[\sum_{G \in \F(P)_p} |(i G)^\circ \cap \Z^n| = \sum_{j=0}^p (-1)^{p-j} \sum_{F \in \F(P)_j} \binom{n-j}{n-p} \ehr_F(i).\]
Applying Equation (\ref{h-st}) to $F \in \F(P)_j$ and noting 
that $\deg(F) \leq \min(j,d)$ by Stanley's monotonicity theorem we get:
\[\sum_{G \in \F(P)_p} |(i G)^\circ \cap \Z^n| 
= \sum_{j=0}^p (-1)^{p-j} \sum_{F \in \F(P)_j} \binom{n-j}{n-p} \left(\sum_{k=0}^{\min(j,d)} h^*_k(F) \binom{i+j-k}{j}\right)\]
\[= \sum_{j=0}^n \sum_{k=0}^{\min(j,d)} 
(-1)^{p-j} \binom{n-j}{n-p} \binom{i+j-k}{j} \left(\sum_{F \in \F(P)_j} h^*_k(F)\right).\] 

Hence, we conclude that Equation (\ref{expr-proof}) equals
\[\sum_{p=d+1}^n \sum_{i=1}^{p-d} (-1)^{d-i} i \binom{p+1}{p-d-i} 
\left(\sum_{j=0}^n \sum_{k=0}^{\min(j,d)} (-1)^{p-j} \binom{n-j}{n-p} \binom{i+j-k}{j} 
\left(\sum_{F \in \F(P)_j} h^*_k(F)\right)\right)\]
\[\sum_{j=0}^n (-1)^j \sum_{k=0}^{\min(j,d)} \left(\sum_{p=d+1}^n \sum_{i=1}^{p-d}(-1)^{p-d-i} i \binom{p+1}{p-d-i} 
\binom{n-j}{n-p} \binom{i+j-k}{j} \right) \left(\sum_{F \in \F(P)_j} h^*_k(F) \right).\]
Let us simplify the expression in the middle bracket:
\[\sum_{p=d+1}^n \sum_{i=1}^{p-d} (-1)^{p-d-i} i \binom{p+1}{p-d-i} \binom{n-j}{n-p} \binom{i+j-k}{j}\]
\[= \sum_{i=0}^{n-d} (-1)^{d-i} i \binom{i+j-k}{j} \left(\sum_{p=d+i}^n (-1)^{p} 
\binom{n-j}{n-p} \binom{p+1}{p-d-i} \right)\]
\[= \sum_{i=0}^{n-d} (-1)^{n-d-i} i \binom{i+j-k}{j} \left(\sum_{q=0}^{n-d-i} (-1)^q 
\binom{n-j}{q} \binom{n+1-q}{n-d-i-q} \right)\]
\begin{equation}
= \sum_{i=0}^{n-d} (-1)^{n-d-i} i \binom{i+j-k}{j} \binom{j+1}{n-d-i},
\label{last}
\end{equation}
where in the last step we used Lemma \ref{binom}, note $n-d-i\geq n-d-(n-d) = 0$. 
Now, we can apply Lemma \ref{formulas}. Let us distinguish the two cases of the theorem:
\begin{enumerate}
\item[{\rm (i)}] If $d < n-d$, then $k \leq \min(j,d) < n-d$, hence expression (\ref{last}) simplifies to $j+1$ by Lemma \ref{formulas}. 
Therefore, Equation (\ref{expr-proof}) equals
\[\sum_{j=0}^n (-1)^j \sum_{k=0}^{\min(j,d)} (j+1) \sum_{F \in \F(P)_j} h^*_k(F) = c(P),\]
where we used $\sum_{k=0}^{\min(j,d)} h^*_k(F) = \Vol_\Z(F)$.
\item[{\rm (ii)}] If $d \geq n-d$, then expression (\ref{last}) simplifies according to the three different ranges of $k$ 
in Lemma \ref{formulas}. Therefore, Equation (\ref{expr-proof}) equals
\[\left(\sum_{j=0}^n (-1)^j \sum_{k=0}^{n-d-1} (j+1) \sum_{F \in \F(P)_j} h^*_k(F)\right) 
+ \sum_{j=0}^n (-1)^j (n-d) \sum_{F \in \F(P)_j} h^*_{n-d}(F).\]
\end{enumerate}

$\hfill$$\qed$

\section{The non-negativity conjecture}
\label{sec:2}

In \cite{DR06} the following conjecture was formulated:

\begin{conjecture}[Di Rocco]{\rm
Let $P \subset \R^n$ be an $n$-dimensional lattice polytope. 
Then $c(P) \geq 0$.}
\label{sandra-conj}
\end{conjecture}

This is not at all obvious or something one would naturally expect from an algebro-geometric viewpoint. 
While in the case of a smooth polytope, $c(P)$ equals the degree of the associated $A$-discriminant $\Delta_A$, 
in the general case $c(P)$ is the degree of a homogeneous rational function $D_A$, called the 
{\em regular $A$-determinant}, see Chapter 11 of \cite{GKZ94}, which does not has to be a polynomial. The
rational function $D_A$ equals the alternating product of the principal determinants
associated to all the facial subsets of $A$.
In Example 2.5 of Chapter 11 of \cite{GKZ94}, an example of a non-simple 
lattice polytope (the hypersimplex $\Delta(3,6)$) is given such that $D_A$ is not a polynomial. Still, Di Rocco 
calculated in \cite{DR06} that $c(\Delta(3,6)) > 0$.

If $P$ is simple and the associated polarized toric variety is {\em quasi-smooth},
it is shown in Theorem 1.6 of Chapter 11 of \cite{GKZ94} 
that $D_A$ is a polynomial (not necessarily equal to $\Delta_A$), so in particular its degree $c(P) \ge 0$.
%(see Chapter 5, Section 4D). 
Here, being quasi-smooth means that %the lattice polytope $P$ is simple and, moreover, 
$i(F,A) = 1$ holds for any face $F$ of $P$, where $i(F,A)$ is the index of the affine lattice 
generated by the lattice points in $F$  with respect to the affine lattice generated by all lattice points in the affine hull of $F$. This condition holds, for instance, for 
any simple polytope which defines a very ample line bundle on the associated $\Q$-factorial toric variety. 
However, Conjecture~\ref{sandra-conj} seems still to be open even for simple polytopes. 

In the following proposition, we verify Conjecture~\ref{sandra-conj} 
for lattice simplices. Note that this case is indeed not covered by the results on quasi-smoothness. 
Starting from dimension four, there are $n$-dimensional lattice simplices that 
are not quasi-smooth: as an example take the convex hull $P$ of 
$e_4$, $e_1 + e_4$, $e_2 + e_4$, $e_1+e_2+2 e_3+e_4$, $-e_1-e_2-e_3-2 e_4$ in $\R^4$, where $e_1, \ldots, e_4$ is a 
lattice basis of $\Z^4$. In this case, the facet $F$ formed by the first four vertices satisfies $i(F,A) = 2$, however we computed that $c(P)  = 21$.

\begin{proposition}\label{simplex}
Let $P \subseteq \R^n$ be an $n$-dimensional lattice simplex. Then
\[c(P) \geq 0.\]
\end{proposition}

\begin{proof}
In the case of a lattice simplex, let us identify $P \subset \R^n$ with 
$P \times \{1\} \subset \R^{n+1}$. Let $v_0, \ldots, v_n$ 
be the vertices of $P$. We identify the faces of $P$ of dimension $j$ with subsets of $[n+1] := \{0, \ldots, n\}$ 
of size $j+1$. As is well-known, the coefficients of the $h^*$-polynomial 
count lattice points in the half-open paralleliped spanned by $P$, see Corollary 3.11 in \cite{BR06}. 
Let $k \in \{0, \ldots, n\}$:
\[h^*_k = |\{x \in \Z^{n+1} \,:\, x = \sum_{i \in [n+1]} \lambda_i v_i,\; 0 \leq \lambda_i < 1,\; 
\sum_{i=0}^n \lambda_i = k\}|.\]
Let us define $s_\emptyset := 1$, and for $\emptyset \not= I \subseteq [n+1]$:
\[s_I := |\{x \in \Z^{n+1} \,:\, x = \sum_{i \in I} \lambda_i (v_i, 1),\; 0 < \lambda_i < 1\}|.\]
Note that 
\[\Vol_\Z(P) = |\{x \in \Z^{n+1} \,:\, x = \sum_{i \in [n+1]} \lambda_i (v_i, 1),\; 0 \leq \lambda_i < 1\} 
= \sum_{I \subseteq [n+1]} s_I.\]
Hence,
\[c(P) = \sum_{j=0}^n (-1)^{n-j} (j+1) \sum_{F \subseteq [n+1], |F| = j+1} \left(\sum_{I \subseteq F} s_I\right).\]
Counting the number of face-inclusions as in the previous proof we get
\[c(P) = (-1)^n \sum_{I \subseteq [n+1]} 
\left(\sum_{j=|I|-1}^n (-1)^j (j+1) \binom{n+1-|I|}{n-j}\right) s_I\]
Let us compute the expression in the bracket using Lemma \ref{binom}:
\[\sum_{j=0}^n (-1)^j (j+1) \binom{n+1-|I|}{n-j} = 
(-1)^n \sum_{q=0}^n (-1)^q \binom{n+1-|I|}{q} \binom{n+1-q}{n-q}\]
\[=(-1)^n \binom{n+1-(n+1-|I|)}{n} = (-1)^n \binom{|I|}{n}.\]
This proves:
\begin{equation}
c(P) = \sum_{I \subseteq [n+1]} \binom{|I|}{n} s_I \geq 0.
\label{noni}
\end{equation}
\end{proof}

\begin{remark}{\rm In the situation of the proof, we can give a more insightful way to see that 
$\cd(P) \geq \frac{n+3}{2}$ implies $c(P) = 0$. For this, observe that Equation (\ref{noni}) implies
\[c(P) = (n+1) s_{[n+1]} + \sum_{I \subseteq [n+1], |I| = n} s_I.\]
In particular, 
\[c(P) = 0 \quad\lolra\quad s_I = 0 \quad\forall\; I \subseteq [n+1], |I| \geq n.\]
By Lemma 2.2 in \cite{Nil08} we have $s_I = 0$, if $|I| > 2 \deg(P)$. Hence, we get
\[\cd(P) \geq \frac{n+3}{2} \lolra n >  2 \deg(P) \lora c(P) = 0.\]
Here is an open question: Does in this case $c(P) =0$ imply that $P$ is a lattice pyramid? This would sharpen the main result in \cite{Nil08}.
}
\end{remark}

One might also wonder whether one could extend the proof of Proposition~\ref{simplex} 
to all lattice polytopes via lattice triangulations. This has been a successful approach to many problems 
in Ehrhart theory, see \cite{BM85, Pay08}.
%There is even an explicit way to compute the $h^*$-polynomial 
%from the $h$-polynomial of the triangulation by counting lattice points in parallelepipeds, 

\section{Directions for future research}

As the reader may have have noticed, we do not have a purely combinatorial proof 
of Corollary~\ref{smoothy}, since we rely on the algebro-geometric 
results in \cite{DR06}. Is there a more direct combinatorial way to deduce the
Cayley structure? 
Moreover, we have to leave the obvious question open, if any of the implications 
\[\cd(P) \geq (n+3)/2 \;\lora\; c(P) =0 \;\lora\; P \text{ Cayley polytope}\]
still hold for arbitrary $n$-dimensional lattice polytopes, cf. Conjecture 4.5 in \cite{DR06}.

A more natural idea is to expect that the following conjecture holds:

\begin{conjecture} {\rm Let $P \subset \R^n$ be an $n$-dimensional lattice polytope. 
with $\cd(P)\geq \frac{n+3}{2}$ (or equivalently, $n > 2 \deg(P)$). Then,
$\deg(\Delta_A) = 0$, that is $X_A$ is dual defect.}
\end{conjecture}

We believe that there should be a 
combinatorial proof in the singular case in a similar fashion as the proof above, using 
the purely combinatorial formulas for the degree of the $A$-discriminant associated to a lattice polytope 
\cite{Est08, MT08}. 

It is already known \cite{CC07, Est08} that if $P$ defines a (non necessarily regular) dual defect 
toric variety, then $P$ has lattice width one, i.e., it is a Cayley polytope of at least 
two lattice polytopes. Moreover, S.~Di Rocco and C.~Casagrande generalized the characterization of dual defect toric manifolds \cite{CS08} 
to the $\Q$-factorial case. There are  known criteria for defectiveness valid in the singular case \cite{DFS07,MT08}, but the precise
combinatorial classification of the structure of lattice configurations $A$ such that $X_A$ is dual defect in
full generality is an open problem.

We end with an example derived from the results of \cite{BDR08} in the classification of
self-dual toric varieties, and results of \cite{CC07}. We construct a
singular dual defect toric variety associated with  a (Cayley) lattice polytope  $P$ of dimension $n=6$ 
which does {\em not} have 
the structure of a Cayley polytope of lattice polytopes $P_0, \ldots, P_k \subseteq \R^m$, where 
$k $ is {\em sufficiently big} compared to $n$ (that is $k$ is not bigger than $\frac{6}{2}=3$). 
Set $m=4$, $k=2$. Let $P_0$ be the triangle with vertices $\{(0,0,0,0), (2,0,0,1), (1,0,0,1)\}$,  $P_1$ the triangle
with vertices $\{(0,0,0,0), (0,2,0,1), (0,1,0,1)\}$, $P_2$ the triangle with vertices $\{ (0,0,0,0),
(0,0,2,1), (0,0,1,0)\}$, and let $P$ be the $6$-dimensional Cayley polytope $P = P_0*P_1*P_2$. 
It follows from \cite{BDR08} that $P$ has exactly $9$ lattice points (which can be easily verified). 
The associated toric variety $(X,L)$ is defective with dual defect equal to $1$ (in fact, it is self-dual).
We thank Andreas Paffenholz for the computation of the codegree of $P$, which is equal to $3$. 
So, ${\rm codeg}(P)$ is \emph{smaller} than $(6+3)/2$. In particular, in the singular case 
defectiveness does not imply high codegree and hence, it does not imply a Cayley structure as in Theorem~\ref{th:main}.

\end{document}